\documentclass{article}

\usepackage{latexsym}
\usepackage{amsmath, amsthm, amssymb, amscd}
\usepackage{graphicx, caption, float}
\usepackage{tikz}
\usepackage{hyperref}
\usepackage{bm, mathrsfs, dsfont}
\usepackage{bbding}
\usepackage[all,pdf]{xy}
\usepackage{upgreek}

\newtheorem{lem}{Lemma}[section]
\newtheorem{prop}{Proposition}[section]
\newtheorem{thm}{Theorem}[section]
\newtheorem{cor}{Corollary}[section]

\theoremstyle{definition}
\newtheorem{ex}{Example}[section]
\newtheorem{ca}{Case}

\title{\bf Constructing Strebel differentials via Belyi maps on the Riemann sphere}
\author{Jijian Song and Bin Xu}

\begin{document}
\maketitle


\noindent{\small {\bf Abstract:} In this manuscript, by using Belyi maps and dessin d'enfants, we construct some concrete examples of Strebel differentials with four double poles on the Riemann sphere. As an application, we could give some explicit cone spherical metrics on the Riemann sphere.


}

\section{Introduction}
Let $X$ be a compact Riemann surface, and let $\Omega_{X}$ denote its cotangent bundle. Then a (meromorphic) {\it quadratic differential} $q$ is a (meromorphic) global section of the line bundle $\Omega_{X}^{\otimes2}$. Let Crit$(q)$ denote the set of zeroes and poles of $q$. Hence $q$ is holomorphic and nowhere vanishing on $X \setminus \text{Crit}(q)$. The restriction of $q$ on $X \setminus \text{Crit}(q)$ defines a conformal flat metric, which is called {\it $q$-metric}. With respect to this metric, a curve $\gamma$ is called a {\it horizontal geodesic} if $q > 0$ along $\gamma$. More precisely, in a coordinate chart $\{U, z\}$ of $X \setminus \text{Crit}(q)$, if $q$ has form $f_{U}(z) dz^{2}$, then the corresponding $q$-metric is $|f_{U}(z)|dz d\bar{z}$ and the horizontal geodesic $\gamma(t)$ satisfies $f_{U}(\gamma(t)) \gamma'(t)^{2} > 0$. A maximal horizontal geodesic is called a {\it horizontal trajectory}. In general, a horizontal trajectory of $q$ may be a closed curve ({\it closed}), or bounded by points in Crit$(q)$ ({\it critical}), or neither ({\it recurrent}).

A quadratic differential $q$ with at most double poles is called {\it Jenkins-Strebel} if the union of all non-closed horizontal trajectories and Crit$(q)$ is compact and of measure zero, or equivalently, $q$ has no recurrent trajectories. Such differentials are first investigated by Jenkins \cite{Jenkins1957} to solve an extremal problem. Later, K. Strebel proves many astonishing results about these quadratic differentials in his famous book \cite{Strebel1984}. One of the most important existence theorems is about a special class of Jenkins-Strebel differentials, which are called {\it Strebel differentials}(see Section \ref{sec:pre} for more details). Arbarello and Cornalba \cite{Arbarello2010} give a directly proof of the existence and uniqueness of Strebel differentials. However, except these general pure existence theorems, it seems that seldom have people talked about how to construct Strebel differentials, not to mention explicit expressions of such differentials.

In contrast to few explicit constructions of Strebel differentials, there are many applications of Strebel differentials in mathematics, such as studying Teichm\"{u}ller theory and the moduli space of pointed compact Riemann surfaces\cite{Loo1995, Masur2009}. Furthermore, Kontsevich\cite{Kon1992,Zvo2004} uses the cell decomposition induced by Strebel differentials to prove Witten's conjecture.

A cone spherical metric is a conformal metric on a compact Riemann surface with constant Gaussian curvature $+1$ and isolated conical singularities. In \cite{Song2017}, we have shown that all periods of Strebel differentials are real. By using this fact, we give a canonical construction of cone spherical metrics by Strebel differentials. In more detail,  suppose $a_{1}, a_{2}, \cdots, a_{n}$ are the residues of a Strebel differential $q$ at $p_{1}, p_{2}, \cdots, p_{n}$ and $m_{1}, m_{2}, \cdots, m_{l}$ are the multiplicities of the zeroes $z_{1}, z_{2}, \cdots, z_{l}$ of $q$. Then the corresponding cone spherical metric represents the divisor
\[ D =  \sum_{i=1}^{n} (a_{i} - 1)p_{i} + \sum_{j=1}^{l}\frac{m_{j}}{2}z_{j}, \]
which is equivalent to that the metric has cone angle $2 \pi a_{i}$ at $p_{i}$ and cone angle $\pi (m_{j} + 2)$ at $z_{j}$, respectively. Hence in order to obtain some explicit cone spherical metrics, we only need to construct some concrete Strebel differentials.

Note that if $q$ is a Strebel differential on $\mathbb{P}^{1}$ and $f: X \rightarrow \mathbb{P}^{1}$ is a branched covering such that the critical values of $f$ are in $\{$critical trajectories of $q\} \cup$Crit$(q)$, then $f^{*}(q)$ is also a Strebel differential on $X$.  Therefore, it is significant to obtain some explicit examples of Strebel differentials on $\mathbb{P}^{1}$. On the other hand, Mulase and Penkava give a construction of a Riemann surface $X$ and a Strebel differential by a metric ribbon graph in \cite{Mulase1998}. In particular, if the metric ribbon graph has rational ratios of the lengths, then the corresponding Strebel differential is a pullback by some Belyi map $f: X \rightarrow \mathbb{P}^{1}$ of the differential
\[ q_0 = -\frac{1}{4 \pi^2} \frac{dz^2}{z(z-1)^2}. \]
However, even on $X = \mathbb{P}^{1}$, we could {\it not} obtain the expressions of Strebel differentials by following their process. The purpose of this manuscript is to present an improvement of that result on $\mathbb{P}^{1}$ for a special case. That is, we will give the explicit expressions of Belyi maps and show that the Belyi maps have minimal degrees.

 We focus in this manuscript on the construction of Strebel differentials with $4$ double poles and residue vector $(1,1,1,1)$ on $\mathbb{P}^{1}$. Let $q$ be a meromorphic quadratic differential on the Riemann sphere with $4$ double poles at $0, 1, \lambda, \infty$ and residue vector $(1,1,1,1)$. Then we can express $q$ as
\begin{align}
\label{QExpression}
 q =q_{\lambda,\,\mu}= -\frac{dz^2}{4\pi^2} \biggl( \frac{1}{z^2} + \frac{1}{(z-1)^2} + \frac{1}{(z-\lambda)^2} + \frac{\mu -2z}{z(z-1)(z-\lambda)} \biggr),
\end{align}
where $\mu \in \mathbb{C}$ is a free complex parameter.
\begin{thm}
\label{thm:DoubleZeroes}
  Suppose that $q$ is a Strebel differential with form \eqref{QExpression} on $\mathbb{P}^{1}$. Then $q$ has either two double zeroes or four simple zeroes. $q=q_{\lambda,\,\mu}$ has two double zeroes if and only if $\lambda \in \mathbb{R} \backslash \{0, 1\}$ and
\begin{displaymath}
   \mu=\mu(\lambda) = \left\{
        \begin{array}{l}
        2\lambda+2, \qquad  \lambda<0;\\
        2 - 2\lambda, \qquad 0<\lambda<1;\\
        2\lambda-2, \qquad \lambda >1.
        \end{array}\right.
\end{displaymath}
In this case, the metric ribbon graph of $q_{\lambda,\,\mu(\lambda)}$ for any $\lambda\in \mathbb{R} \backslash \{0, 1\}$ can be realized by some $\lambda_0 \in [\frac{1}{2},1)$.
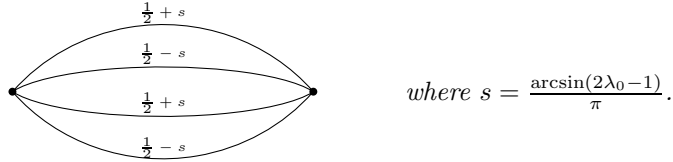
\begin{figure}[H]
  \begin{center}
    \begin{tikzpicture}
      \draw (2,0) arc (20:160:2.13cm and 0.5cm)
            (-2,0) arc (200:340:2.13cm and 0.5cm)
            (2,0) arc (40:140:2.61cm and 2.5cm)
            (-2,0) arc (220:320:2.61cm and 2.5cm);
      \fill (2,0) circle (1.5pt)
            (-2,0) circle (1.5pt);
      \node at (0,0.5) {\tiny{$\frac{1}{2}-s$}};
      \node at (0,-0.15) {\tiny{$\frac{1}{2}+s$}};
      \node at (0,1.05) {\tiny{$\frac{1}{2}+s$}};
      \node at (0,-0.75) {\tiny{$\frac{1}{2}-s$}};

      \node at (5,0) {\it where $s= \frac{\arcsin(2\lambda_0-1)}{\pi}$.};
    \end{tikzpicture}
  \end{center}
 \caption{\it The corresponding metric ribbon graphs if $\lambda_0 \in [\frac{1}{2},1)$.}
\end{figure}

\end{thm}

As a consequence, suppose the residues of the Strebel differential $q$ are all equal. Then $q$ has $4$ simple zeroes if and only if $\lambda \in \mathbb{C \setminus R}$,  i.e., the double poles of $q$ are non-coaxial. Unfortunately,  we have not yet obtained all the explicit expressions of $q$ in this general case. Denote by $q \sim q'$ if two differentials $q$ and $q'$ coincide  up to a non-zero constant multiple. Then we have
\begin{thm}
\label{thm:GeneralCase}
  Let $q$ be a Strebel differential on $\mathbb{P}^{1}$ with four simple zeroes and residue vector $(1,1,1,1)$. Then the metric ribbon graph of $q$ coincides with the following graph
  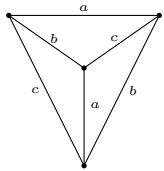
\begin{figure}[H]
  \begin{center}
    \begin{tikzpicture}
      \draw (0,0.3)--(1,1)
            (0,0.3)--(0,-1)
            (0,0.3)--(-1,1)
            (1,1)--(0,-1)
            (0,-1)--(-1,1)
            (-1,1)--(1,1);
       \fill(0,0.3) circle (1pt)
            (1,1) circle (1pt)
            (0,-1) circle (1pt)
            (-1,1) circle (1pt);
      \node at (0,1.1) {\tiny{$a$}};
      \node at (0.15,-0.2) {\tiny{$a$}};
      \node at (-0.4,0.7) {\tiny{$b$}};
      \node at (0.65,0) {\tiny{$b$}};
      \node at (0.4,0.7) {\tiny{$c$}};
      \node at (-0.65,0) {\tiny{$c$}};
    \end{tikzpicture}
  \end{center}
  \caption{\it Metric ribbon graphs with residues $(1,1,1,1)$ and $4$ simple zeroes.}
  \label{fig:Generalcase}
  \end{figure}
\noindent where $a,b,c>0$ and $a+b+c=1$. Suppose that $a,b,c \in \mathbb{Q}_{>0}$ and $a+b+c=1$. Then there exists a Belyi map $f$ such that $f^*(q_0) \sim q$ and $f$ could be decomposed to be $f = g \circ x^2$ with another Belyi map $g$. Furthermore, if $d$ is the minimal positive integer such that $da,\, db$ and $dc$ are all integers, then
    \[\min\,\left\{\deg f \mid \text{$f$ is a Belyi map and\ }f^*(q_0)\sim q\right\}=\begin{cases}
    2d &\text{if}\ 2\mid d,\\
    4d &\text{if}\ 2\nmid d,
    \end{cases}\]
and we obtain the explicit expressions of Belyi maps and Strebel differentials when $(a,b,c)$ equals one of the following five triples{\rm :}
\[ \Big(\frac{1}{2}, \frac{1}{4}, \frac{1}{4}\Big), \Big(\frac{1}{3}, \frac{1}{3}, \frac{1}{3}\Big), \Big(\frac{1}{3}, \frac{1}{6}, \frac{1}{2}\Big), \Big(\frac{1}{3}, \frac{1}{2}, \frac{1}{6}\Big), \Big(\frac{2}{3}, \frac{1}{6}, \frac{1}{6}\Big). \]
which exhaust all the possibilities such that the Belyi maps have minimal degrees equal to either $8$ or $12$.
\end{thm}
In \cite{Mulase2001},  Mulase and Penkava  conjectured that {\it if there exists a Strebel differential $q$ on $X$ such that all lengths of critical trajectories of $q$ are algebraic but not rational under $q$-metric, then the pointed Riemann surface $\big(X, (p_{1}, p_{2}, \cdots, p_{n})\big)$ could not be defined over $\overline{\mathbb{Q}}$}. The examples of Strebel differentials we construct provide more evidences for this conjecture. As an application on cone spherical metrics, we have
\begin{cor}
\label{cor:metric}
The moduli space of cone spherical metrics with four conical singularities of angles $3\pi, 3\pi, 3\pi, 3\pi$ on $\mathbb{P}^{1}$ has a subspace homeomorphic to the quotient space of the triangle region $\{(a,b,c)\in {\Bbb R}\mid \, a,b,c>0,\, a+b+c=1\}$ by the group $\mathbb{Z}/3\mathbb{Z}$ generated by the cyclic transformation $(a,b,c)\mapsto (b,c,a)$. 
\end{cor}

The organization of this manuscript is as follows. In Section \ref{sec:pre}, for the convenience of readers, we recall in detail the existence theorem of Strebel differentials and the correspondence between Strebel differentials and metric ribbon graphs. As an application, we give the proof of Corollary \ref{cor:metric}. The proof of Theorem \ref{thm:DoubleZeroes} occupies the whole of Section \ref{sec:2zeroes}. We prove Theorem \ref{thm:GeneralCase} in Section \ref{sec:4zeroes} .

\section{Preliminaries}
\label{sec:pre}
In this section, we will recall some basic results about Strebel differentials, such as the existence theorem/definition given by K. Strebel, Harer's one-to-one correspondence between pointed compact Riemann surfaces with Strebel differentials and metric ribbon graphs. For more details, one can see \cite{Mulase1998}.

\begin{thm}
\label{thm:Strebel}
{\rm (\cite[Theorem 23.5]{Strebel1984})}
Let $X$ be a compact Riemann surface of genus $g$ with $n$ marked points $p_1, p_2, \cdots, p_n$, and $a_1,\cdots, a_n \in \mathbb{R}_{>0}$. If $2 - 2g - n < 0$, then there exists a unique quadratic differential $q \in H^{0}(X, \Omega_{X}^{\otimes}(2p_{1} + 2p_{2} + \cdots + 2p_{n}))$ such that
\begin{enumerate}
\item $p_{i}$ is a double pole with residue $a_{i}$ of $q$ for $i = 1, 2, \cdots, n$.
\item The union of all non-closed trajectories is a set of measure zero.
\item Every closed trajectory is a circle around some $p_{i}$.
\end{enumerate}
\end{thm}
Then the quadratic differential $q$ is called a {\it Strebel differential}. In \cite{Strebel1984},  K. Strebel also proved that the closure of any recurrent trajectory is a subset of $X$ of positive measure. Hence, the second condition in Theorem \ref{thm:Strebel} is equivalent to say $q$ has no recurrent trajectories. If $\{U, z\}$ is a local coordinate around $p_{i}$ with $z(p_{i}) = 0$, then the expression of $q$ on $U$ is
\[ \Big( -\frac{a_{i}^{2}}{z^{2}} + \frac{b_{i}}{z} + h_{i}(z) \Big)dz^{2}, \]
where $h_{i}(z)$ is a holomorphic function on $U$. By the third condition, we know that, for each $p_{i}$, the union of all closed trajectories around $p_{i}$ is an open punctured disc and $p_{i}$ is the center of the disc.

At a zero of multiplicity $m$ of $q$, there are $m + 2$ half critical trajectories emanating from the zero. Moreover, the critical graph constituted by critical trajectories and the zeroes of $q$ is connected. Each edge of the critical graph has a length measured by $q$-metric. Hence we obtain a connected metric graph $\Gamma$ drawn on $X$, which is called a {\it metric ribbon graph}. Moreover,  the cell decomposition of $X$ induced by $\Gamma$ has $n$ discs. The number $n$ is called the {\it number of boundary components} of $\Gamma$. Note that, at any vertex $v$ of $\Gamma$,  the orientation of $X$ induces a cyclic ordering of the half edges incident to $v$. As an example, a metric ribbon graph on $\mathbb{P}^{1} = \mathbb{C} \cup \{ \infty \}$ is nothing but a planar metric graph with natural cyclic ordering at each vertex.

Given a metric ribbon graph $\Gamma$,  Mulase and Penkava \cite[Theorem 5.1]{Mulase1998} proved that there exists a Riemann surface $X$ and a Strebel differential $q$ on it such that $\Gamma$ coincides with the critical graph of $q$. Therefore, there exists a correspondence between Riemann surfaces with Strebel differentials and metric ribbon graphs. Furthermore, Harer \cite{Harer1988} proved that this correspondence is actually an orbifold isomorphism
\begin{align}
\label{HCorrespondence} 
\begin{split}
  \mathfrak{M}_{g,n} \times \mathbb{R}^n_{>0} \rightarrow \coprod_{\Gamma} \frac{\mathbb{R}^{e(\Gamma)}_{>0}}{{\rm Aut}_{\partial}(\Gamma)}, \\
   (X,(p_1, p_2, \cdots, p_n)) \times (a_1, a_2, \cdots, a_n) \mapsto   \Gamma.
\end{split}
\end{align}
where $\mathfrak{M}_{g,n}$ is the moduli space of compact Riemann surfaces of genus $g$ with $n$ marked ordered points, $\Gamma$ runs over all ribbon graphs with degree of each vertex $\geq 3$ and with $n$ boundary components, and ${\rm Aut}_{\partial}(\Gamma)$ is the automorphism group of the ribbon graph $\Gamma$. Then combining the correspondence \eqref{HCorrespondence} and Theorem \ref{thm:GeneralCase}, we could give the proof of Corollary \ref{cor:metric}.

\begin{proof}[Proof of Corollary \ref{cor:metric}]
Let $S$ denote the set of all Strebel differentials on $\mathbb{P}^{1}$ with expressions \eqref{QExpression} and $4$ simple zeroes.  For any $q \in S$, by the construction in \cite{Song2017}, we could obtain a cone spherical metric representing the divisor $D = \sum_{j=1}^{4}\frac{1}{2}z_{j}$, i.e., the metric has singular angles $3\pi, 3\pi, 3\pi, 3\pi$ at $z_{1}, z_{2}, z_{3}, z_{4}$. Suppose Strebel differentials $q_{1}, q_{2} \in S$ have the same zero points. Then $q_{1} = q_{2}$ by the expression \eqref{QExpression}. Hence, we always obtain different spherical metrics from distinct Strebel differentials in $S$. By the correspondence \eqref{HCorrespondence}, There exists a bijective correspondence between $S$ and $\{(a,b,c)\in \mathbb{R}_{>0} \mid a+b+c=1\}/{\rm Aut}_{\partial}(\Gamma)$, where $\Gamma$ is the underlying ribbon graph in Figure \ref{fig:Generalcase}. Note that the automorphism group of $\Gamma$ is $\mathbb{Z}/3\mathbb{Z}$(see \cite[Definition 1.8]{Mulase1998}). Hence, we are done.
\end{proof}

\section{Strebel differentials with two double zeroes}
\label{sec:2zeroes}
In this section, we give all the expressions of Strebel differentials whose zero partitions are $4 = 2 + 2$(i.e. $2$ double zeroes) and residue vector $(1,1,1,1)$. From now on, we fix the underlying compact Riemann surface $X = \mathbb{P}^{1}$.

The strategy of our construction is to study the holomorphic maps $f$ from $\mathbb{P}^{1}$ to $\mathbb{P}^{1}$ of degree $4$ such that  $f^*q_0$ has $4$ double poles. Firstly, we prove that the zero partition of $f^*q_0$ can only be $2+2$ or $1+1+1+1$. Secondly, we give the expression of Strebel differential for $\lambda= \frac{1}{2}$ by writing down $f$ with only $2$ critical points. Thirdly, through the research of the branched covering $f$ with $3$ critical points and the critical graph of $f^*q_0$, we work out all the expressions of Strebel differentials if $\lambda \in (\frac{1}{2},1)$. Then by considering M\"{o}bius transformations $x \mapsto 1-x$ and $x \mapsto \frac{1}{x}$, we obtain all the expressions of Strebel differentials for $\lambda \in \mathbb{R} \backslash \{0, 1\}$. At last, we show that these differentials are all the Strebel differentials with residue vector $(1,1,1,1)$ and $2$ double zeroes by investigating the corresponding metric ribbon graphs.

Suppose $q_1$ and $q_2$ have $4$ double poles at the same points with the same residues. Then $q_1-q_2$ has only simple poles. Let $D$ be a divisor of degree $4$ on $\mathbb{P}^{1}$, we know that $\dim_{\mathbb{C}}H^0(K^2(D)) = 1$. Hence the meromorphic quadratic differentials which have $4$ double poles at $(0,1,\lambda,\infty)$ with residue vector $(1,1,1,1)$ have the form of
\begin{footnotesize}
  \begin{align*}
    q &= -\frac{dz^2}{4\pi^2}[\frac{1}{z^2} + \frac{1}{(z-1)^2} + \frac{1}{(z-\lambda)^2} + \frac{\mu-2z}{z(z-1)(z-\lambda)} ] \\
      &= -\frac{dz^2}{4\pi^2} \frac{z^4 + (\mu-2(\lambda+1))z^3+(2(\lambda^2+\lambda+1)-\mu(\lambda+1))z^2 + (\lambda\mu-2\lambda(\lambda+1))z + \lambda^2}{z^2(z-1)^2(z-\lambda)^2},
  \end{align*}
\end{footnotesize}
where $\mu \in \mathbb{C}$ is a parameter.

For the zero partition of the quadratic differential $q$ with residue vector $(1,1,1,1)$ on the Riemann sphere, we have the following property:
\begin{lem}
\label{lem:ZeroPartition}
  Let $q$ be a quadratic differential on $\mathbb{P}^{1}$ with $4$ double poles and residue vector $(1,1,1,1)$. Then the zero partition of $q$ is either $4 = 2+2$ or $4=1+1+1+1$, i.e. $q$ has $2$ double zeroes or $4$ simple zeroes.
\end{lem}
\begin{proof}

In order to investigate the zero partition of $q$, we only need consider the numerator of the expression of $q$. The discriminant of the polynomial
\begin{footnotesize}
\begin{align}
\label{discriminant}
z^4 + (\mu-2(\lambda+1))z^3+(2(\lambda^2+\lambda+1)-\mu(\lambda+1))z^2 + (\lambda\mu-2\lambda(\lambda+1))z + \lambda^2
\end{align}
\end{footnotesize}
is
$$\lambda^2 (\lambda-1)^2 (\mu -2+2\lambda)^2 (\mu -2-2\lambda)^2 (\mu+2-2\lambda)^2.$$
Hence, $q$ has a multiple zero if and only if $\mu = 2(1-\lambda)$, $2(1+\lambda)$ or $2(\lambda-1)$. For these three cases,  the corresponding expressions of (\ref{discriminant}) are $(z^2- 2\lambda z + \lambda)^2$, $(z^2-\lambda)^2$ or $(z^2-2z+\lambda)^2$ respectively. As a result, $q$ has either $4$ simple zeroes or $2$ double zeroes.
\end{proof}
Therefore, if $q$ is a Strebel differential on the Riemann sphere with $4$ double poles and residue vector $(1,1,1,1)$, then the zero partition of $q$ can only be $4=2+2$ or $4=1+1+1+1$. For these two partitions, we can determine their ribbon graphs
\begin{lem}
\label{lem:PossibleGraphs}
   Suppose $q$ is a Strebel differential with residue vector $(1,1,1,1)$. If the zero partition of $q$ is $4=2+2$, then its ribbon graph looks like
  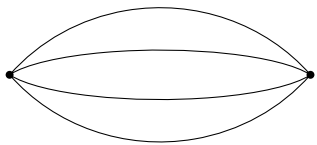
\begin{figure}[H]
  \begin{center}
    \begin{tikzpicture}
      \draw (2,0) arc (20:160:2.13cm and 0.5cm)
            (-2,0) arc (200:340:2.13cm and 0.5cm)
            (2,0) arc (40:140:2.61cm and 2.5cm)
            (-2,0) arc (220:320:2.61cm and 2.5cm);
      \fill (-2,0) circle (1.5pt)
           (2,0)  circle (1.5pt);
  \end{tikzpicture}
  \end{center}
\caption{Ribbon graph for the partition $4=2+2$.}
\label{fig:2DoubleZeroes}
\end{figure}
If the zero partition of $q$ is $4=1+1+1+1$, its ribbon graph is
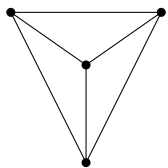
\begin{figure}[H]
\begin{center}
    \begin{tikzpicture}
      \draw (0,0.3)--(1,1)
            (0,0.3)--(0,-1)
            (0,0.3)--(-1,1)
            (1,1)--(0,-1)
            (0,-1)--(-1,1)
            (-1,1)--(1,1);
       \filldraw (0,0.3) circle (1.5pt)
                 (0,-1) circle (1.5pt)
                 (1,1) circle (1.5pt)
                 (-1,1) circle (1.5pt);
    \end{tikzpicture}
  \end{center}
\caption{Ribbon graph for the partition $4=1+1+1+1$.}
\label{fig:4SimpleZeroes}
\end{figure}
\end{lem}
\begin{proof}
   For the first case, let $\Gamma$ be the metric ribbon graph of $q$. Suppose there  exists a loop $l$ in $\Gamma$.  Then the Riemann sphere is divided into $2$ regions by $l$. Let $D$ be the one of the $2$ regions containing no vertex of $\Gamma$. Since the total length of the boundary of each component of $\Gamma$ is $1$, we conclude that there is no other loop in the interior of $D$, which means that $D$ is a component of $\Gamma$ and the length of $l$ is $1$. Then the length of boundaries of the other component besides $D$ touched by $l$ is greater than $1$, contradiction! Hence, $\Gamma$ is a planar graph with $2$ vertices and no loop. The only possible graph is Figure \ref{fig:2DoubleZeroes} since the degree of each vertex is $4$.

    For the second case, we also show that there is no loop in $\Gamma$. Otherwise,  note that we can assume $D$ contains at most one vertex of $\Gamma$. If there is no vertex in $D$, the argument is the same as the first case. Suppose there is a vertex $v$ in $D$. Then there is a small loop in $D$ incident to $v$ with the length of $1$, contradiction! Then $\Gamma$ can only be Figure \ref{fig:4SimpleZeroes} and the following one
\begin{figure}[H]
\begin{center}
\begin{tikzpicture}
  \draw (1,0) arc (20:160:1.06cm and 0.25cm)
            (-1,0) arc (200:340:1.06cm and 0.25cm)
            (1,1) arc (20:160:1.06cm and 0.25cm)
            (-1,1) arc (200:340:1.06cm and 0.25cm)
            (-1,1)--(-1,0)
            (1,1)--(1,0);
\end{tikzpicture}
\end{center}
\caption{A fake ribbon graph.}
\label{fig:fake}
\end{figure}
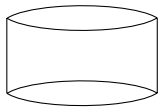
However, for Figure \ref{fig:fake}, it can not be a metric ribbon graph with residue $(1,1,1,1)$.
\end{proof}

Now, we give a simple meromorphic quadratic differential which plays the role of a building block in the following construction (see \cite{Mulase1998} Example 4.4).
\begin{ex}
  \label{ex:block}
  Consider the meromorphic quadratic differential on $\mathbb{P}^{1}$
  \begin{equation*}
    q_0'=\frac{1}{4\pi^2}\frac{dz^2}{z(1-z)}.
  \end{equation*}
  It has simple poles at $0$ and $1$, and a double pole at $\infty$. By solving differential equations, we know that the line segment $[0,1]$ is a horizontal trajectory of length $1/2$. The whole $\mathbb{P}^{1}$ minus $[0,1]$ and $\infty$ is covered by a collection of compact horizontal trajectories which are confocal ellipses
  \begin{equation*}
    z=a \cos \theta +1/2 + i b \sin\theta,
  \end{equation*}
  where $a$ and $b$ are positive constants that satisfy $a^2 = b^2 + 1/4$. The length of each closed curve is $1$.
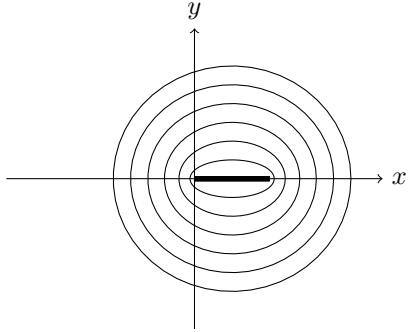
\begin{figure}[H]
\begin{center}
  \begin{tikzpicture}[domain=0:2*pi]
    \draw[->] (-2.5,0) -- (2.5,0) node[right] {$x$};
    \draw[->] (0,-2) -- (0,2) node[above] {$y$};
    \draw[line width=2pt] plot ({0.5*cos(\x r)+0.5},0);
    \draw[samples=50] plot ({0.56*cos(\x r)+0.5},{0.25*sin(\x r)});
    \draw[samples=50] plot ({0.707*cos(\x r)+0.5},{0.5*sin(\x r)});
    \draw[samples=50] plot ({0.9*cos(\x r)+0.5},{0.75*sin(\x r)});
    \draw[samples=50] plot ({1.12*cos(\x r)+0.5},{1.0*sin(\x r)});
    \draw[samples=50] plot ({1.35*cos(\x r)+0.5},{1.25*sin(\x r)});
    \draw[samples=50] plot ({1.58*cos(\x r)+0.5},{1.5*sin(\x r)});
  \end{tikzpicture}
\end{center}
\caption{Horizontal trajectories of $q_0'$.}
\label{fig:BuildingBlock}
\end{figure}

\end{ex}

Let $\phi (z)=\frac{z}{z-1}$ and $q_0=\phi ^*(q_0') = -\frac{1}{4 \pi^2} \frac{dz^2}{z(z-1)^2}$. Then $q_0$ has simple poles at $0, \infty$, and a double pole at $1$ with residue $1$. For any compact Riemann surface $X$ and a holomorphic map $f: X \rightarrow \mathbb{P}^{1}$, $f^*q_0$ has only finite critical trajectories and no recurrent trajectories since $f$ is proper.

By Theorem \ref{thm:Strebel}, we know that, for any given $\lambda \in \mathbb{C}\backslash \{0,1\}$, the value of $\mu$ in (\ref{QExpression}) is unique if $q$ is Strebel. In order to get the value $\mu(\lambda)$ for some $\lambda$, let us consider $f: \mathbb{P}^{1} \rightarrow \mathbb{P}^{1}$ be a branched covering such that $f^*(q_0)$ has $4$ double poles with residue vector $(1,1,1,1)$ and no simple poles on $\mathbb{P}^{1}$, then
\begin{itemize}
  \item $\deg f=4$;
  \item $1$ is not the critical value of $f$;
  \item the local ramification degrees $>1$ over $0$ and $\infty$.
\end{itemize}
By Riemann-Hurwitz formula, the total branch order $\nu (f)=6$. We consider the following two cases:

\begin{ca}[{\bf The expression of the Strebel differential for $\lambda = \frac{1}{2}$}]
\label{ca:1}
\quad \\
If the local ramification degrees over $0$ and $\infty$ are both $4$, we can assume that $f(0)=0, f(\infty)=\infty$ and $f$ has the form of $cx^4$. For distinct $c$, the sets $\{0,\infty\} \cup \{$roots of $f(x)=0\}$ are all equivalent under M\"{o}bius transformation. Hence, we only need to consider $f(x)=x^4$, then
  \begin{align*}
    f^*(q_0)=-\frac{dx^2}{4\pi^2}\frac{16x^2}{(x^2+1)^2 (x^2-1)^2},
  \end{align*}
  and its critical graph is
 \begin{figure}[H]
  \begin{center}
    \begin{tikzpicture}
      \draw (2,0) arc (20:160:2.13cm and 0.5cm)
            (-2,0) arc (200:340:2.13cm and 0.5cm)
            (2,0) arc (40:140:2.61cm and 2.5cm)
            (-2,0) arc (220:320:2.61cm and 2.5cm);
      \node at (2.2,0) {\tiny{$\infty$}};
      \node at (-2.1,0) {\tiny{$0$}};
      \node at (0,0.5) {\tiny{$\frac{1}{2}$}};
      \node at (0,-0.15) {\tiny{$\frac{1}{2}$}};
      \node at (0,1.05) {\tiny{$\frac{1}{2}$}};
      \node at (0,-0.75) {\tiny{$\frac{1}{2}$}};
    \end{tikzpicture}
  \end{center}
 \caption{Metric critical graph of $f^*q_0$.}
 \label{fig:GraphDegree4}
\end{figure}
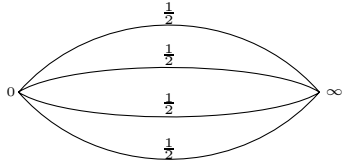
  Since $f^*(q_0)$ has only $4$ critical trajectories and its critical graph is connected, it is a Strebel differential on $\mathbb{P}^{1}$. Consider M\"{o}bius transformation $\varphi (x)=\frac{1-(1-i)x}{-1+(1+i)x}$, then
  \begin{align*}
    \varphi^*f^*q_0=-\frac{dx^2}{4\pi^2}\frac{x^4 - 2x^3 + 2x^2 - x +1/4}{x^2 (x-1)^2 (x-1/2)^2},
  \end{align*}
  which is a Strebel differential with four double poles at $(0,1,1/2,\infty)$ and residue vector $(1,1,1,1)$.
\end{ca}

\begin{ca}[{\bf The expressions of Strebel differentials for $\lambda \in (\frac{1}{2}, 1)$}]
\label{ca:2}
\quad \\
If the local ramification degrees over $0$ and $\infty$ are $(2,2)$ and $4$ respectively, we can assume $f(x)=\frac{1}{c^2}x^2 (x-1)^2$ with $c \in \mathbb{C}^*$. The zeroes of $f^*q_0$ have the type of multiplicities $(2,2)$ and the double poles are located at
\begin{align*}
  \frac{1+\sqrt{1+4c}}{2}, \frac{1-\sqrt{1+4c}}{2}, \frac{1+\sqrt{1-4c}}{2}, \frac{1-\sqrt{1-4c}}{2}.
\end{align*}
Taking a M\"{o}bius transformation
\begin{align*}
  \varphi (z) = \frac{z-\frac{1+\sqrt{1+4c}}{2}}{z-\frac{1+\sqrt{1-4c}}{2}} \cdot \frac{\frac{1-\sqrt{1+4c}}{2}-\frac{1+\sqrt{1-4c}}{2}}{\frac{1-\sqrt{1+4c}}{2}-\frac{1+\sqrt{1+4c}}{2}},
\end{align*}
we have
\begin{align*}
  \varphi \left( \frac{1- \sqrt{1 - 4c}}{2} \right) = \frac{1+\sqrt{1 - 16c^2}}{2\sqrt{1 - 16c^2}},
\end{align*}
i.e. the M\"{o}bius transformation $\varphi$ sends the location of four double poles to $(0,1,\infty, \lambda)$.
\begin{align*}
  \lambda = \frac{1+ \sqrt{1- 16c^2}}{2\sqrt{1- 16c^2}}, \\
  c^2 = \frac{\lambda (\lambda-1)}{4 (2\lambda - 1)^2}.
\end{align*}
A routine computation gives rise to $f'(x)=\frac{2}{c^2}x(x-1)(2x-1)$. Thus the ramification points of $f$ are $0,1,\frac{1}{2},\infty$ and
\begin{align*}
  f(0)=f(1)=0, \\
  f(\infty)=\infty, \\
  f(\frac{1}{2})=\frac{1}{16c^2}.
\end{align*}
Hence, for any $c^2 \neq \frac{1}{16}$, $f^*q_0$ has $4$ double poles with residue vector $(1,1,1,1)$ and $2$ double zeroes at $\frac{1}{2}, \infty$. 
If $f(\frac{1}{2}) = \frac{1}{16c^2} \in (-\infty,0)$, the horizontal trajectories of $f^*q_0$ are as follows
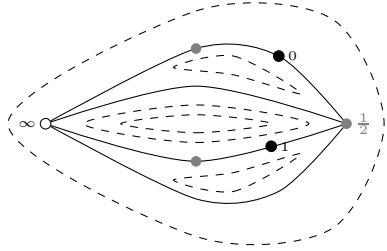
\begin{figure}[H]
\begin{center}
  \begin{tikzpicture}
    \draw plot[smooth] coordinates {(-2,0) (0,0.5) (2,0)};
    \draw plot[smooth] coordinates {(-2,0) (0,-0.5) (2,0)};
    \draw[dashed] plot[smooth cycle] coordinates {(-1.5,0) (0,-0.25) (1.5,0) (0,0.25)};
    \draw[dashed] plot[smooth cycle] coordinates {(-1,0) (0,-0.12) (1,0) (0,0.12)};
    \draw plot[smooth] coordinates {(-2,0) (0,1) (1.1,0.9) (2,0)};
    \draw plot[smooth] coordinates {(-2,0) (0,-1) (1.1,-0.9) (2,0)};
    \draw[dashed] plot[smooth cycle] coordinates {(-0.3,0.75) (0.5,0.9) (1.4,0.4) (0.5,0.65)};
    \draw[dashed] plot[smooth cycle] coordinates {(-0.3,-0.75) (0.5,-0.9) (1.4,-0.4) (0.5,-0.65)};
    \draw[dashed] plot[smooth cycle] coordinates {(-2.5,0) (0,1.5) (1.8,1.4) (2.5,0) (1.8,-1.4) (0,-1.5)};
    \filldraw[white] (-2,0) circle (2pt);
    \draw (-2,0) circle (2pt) node[left] {\tiny{$\infty$}};
    \fill[gray] (0,1) circle (2pt)
                (0,-0.5) circle (2pt)
                (2,0) circle (2pt) node[right] {\tiny{$\frac{1}{2}$}};
    \filldraw (1.1,0.9) circle (2pt) node[right] {\tiny{$0$}}
              (1,-0.3) circle (2pt) node[right] {\tiny{$1$}};
  \end{tikzpicture}
\end{center}
\caption{\footnotesize{Horizontal trajectories of Strebel differential with type of zeroes $(2,2)$.}}
\label{fig:DoubleZeroes}
\end{figure}
Therefore, $f^*q_0$ is Strebel differential when $c^2 < 0$. In what follows, we consider $c^2 < 0$, i.e. $\lambda \in (\frac{1}{2},1)$, then
\begin{align*}
  f^*q_0 = -\frac{dx^2}{4\pi^2} \frac{4c^2(2x-1)^2}{(x^2-x-c)^2 (x^2-x+c)^2}.
\end{align*}
The inverse transformation of $\varphi(z)$ is
\begin{align*}
  2z &= \frac{2(1+\sqrt{1-4c})x - \Big(1+\sqrt{1+4c}+\sqrt{1-4c}+\sqrt{\frac{1-4c}{1+4c}}\Big)}{2x - (1+\sqrt{\frac{1-4c}{1+4c}})}\\
     &= (1+\sqrt{1-4c}) - \frac{\frac{8c}{\sqrt{1+4c}}}{2x - \Big(1+\sqrt{\frac{1-4c}{1+4c}}\Big)}.
\end{align*}
By a direct calculation, we have
\begin{align*}
  (\varphi^{-1})^*f^*q_0 = -\frac{dx^2}{4\pi^2} \frac{\Big(x-\frac{1}{2}(1+\sqrt{\frac{1-4c}{1+4c}})\Big)^2 \Big(x-\frac{1}{2}(1+\sqrt{\frac{1+4c}{1-4c}})\Big)^2}{x^2 (x-1)^2 \Big(x-\frac{1+\sqrt{1-16c^2}}{2\sqrt{1-16c^2}}\Big)^2}.
\end{align*}
Therefore
\begin{align*}
  \lambda = \frac{1+\sqrt{1-16c^2}}{2\sqrt{1-16c^2}}, \\
  \mu (\lambda) = 2-2\lambda.
\end{align*}
\end{ca}

\begin{proof}[Proof of Theorem \ref{thm:DoubleZeroes}]
By Case \ref{ca:1} and \ref{ca:2}, we know the expressions of Strebel differentials when $\lambda \in [\frac{1}{2},1)$. In order to obtain all the expressions of Strebel differentials for $\mathbb{R} \backslash \{0, 1\}$, we consider M\"{o}bius transformation $x \mapsto 1-x$, then $q$ becomes to
\begin{align*}
  -\frac{dx^2}{4\pi^2}\Big(\frac{1}{x^2} + \frac{1}{(x-1)^2} + \frac{1}{(x-(1-\lambda))^2} + \frac{2-\mu(\lambda) - 2x} {x (x-1) (x-(1-\lambda))}\Big).
\end{align*}
Hence $\mu(1-\lambda) = 2-\mu(\lambda)$. By considering $x \mapsto 1/x$, we get $\mu(1/\lambda) = \mu(\lambda)/\lambda$. To sum up all results, $q$ is Strebel differential if $\mu$ and $\lambda$ satisfy
\begin{figure}[H]
\begin{center}
  \begin{tikzpicture}[>=stealth]
    \draw[->] (-3.5,0) -- (3.5,0) node[right] {$\lambda$};
    \draw[->] (0,-3) -- (0,3) node[above] {$\mu$};
    \foreach \x in {-3,-2,-1,1,2,3} {\draw (\x,-0.05) -- (\x,0.05) node[below] {\tiny{\x}};}  
    \foreach \y in {-2,-1,1,2} {\draw (-0.05,\y) -- (0.05,\y) node[left] {\tiny{\y}};}  
    \node at (0.1,-0.15) {\tiny{$0$}};
    \draw (2.3,2.6) -- (1,0) -- (0,2) -- (-2.2,-2.4);
    \draw[dashed] (0,2) -- (2,2) -- (2,0);
    \fill[white] (0,2) circle (1pt)
                 (1,0) circle (1pt);
    \draw (0,2) circle (1pt)
          (1,0) circle (1pt);
    \node at (2.6,2.2) {\tiny{$2\lambda -2$}};
    \node at (-1,1.1) {\tiny{$2\lambda +2$}};
  \end{tikzpicture}
\end{center}
\caption{The relation between $\mu$ and $\lambda$ when $q$ is Strebel.}
\end{figure}
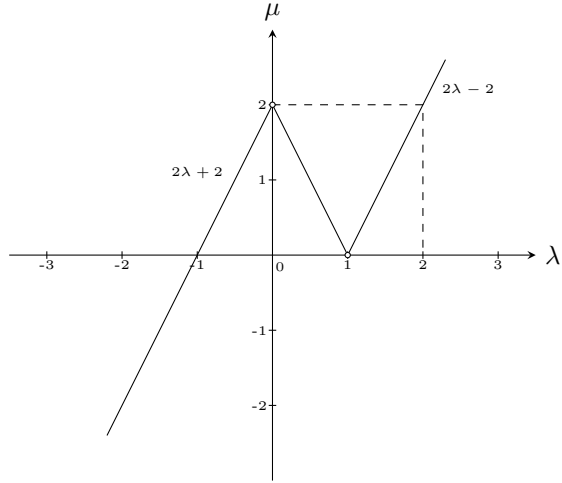
where $q = -\frac{dz^2}{4\pi^2}\Big(\frac{1}{z^2} + \frac{1}{(z-1)^2} + \frac{1}{(z-\lambda)^2} + \frac{\mu-2z}{z(z-1)(z-\lambda)}\Big)$.

Now consider $\lambda \in [\frac{1}{2},1)$, then $\mu(\lambda) = 2-2 \lambda$ and the length of the arc from $0$ (or $1$) to $\frac{1}{2}$ in Figure \ref{fig:DoubleZeroes} is
\begin{align*}
\Big|\int_{f(0)}^{f(\frac{1}{2})} \sqrt{q_0}\Big| &= \Big|\int_0^{\frac{1}{16c^2}} \sqrt{q_0}\Big| \\
                                                                     &=  \Big|\int_0^{\frac{1}{1-16c^2}} \sqrt{q'_0}\Big| \\
                                                                     &= \frac{1}{2\pi} \int_0^{(2\lambda-1)^2} \frac{dz}{\sqrt{z(1-z)}}\\
                                                                     &=  \frac{1}{\pi} \int_0^{\arcsin(2\lambda-1)} d\theta \quad ( z=\sin ^2 \theta) \\
                                                                     &= \frac{\arcsin(2\lambda-1)}{\pi},
\end{align*}
which means that we obtain all the metric ribbon graphs as Figure \ref{fig:DoubleZeroes}. Since the metric ribbon graphs of Strebel differentials with $2$ double zeroes and residue vector $(1,1,1,1)$ are exhausted by Figure \ref{fig:DoubleZeroes}, we complete our proof by Lemma \ref{lem:PossibleGraphs} and Harer's correspondence.
\end{proof}

\section{Strebel differentials with four simple zeroes}
\label{sec:4zeroes}
Let $q$ be a Strebel differential with residue vector $(1,1,1,1)$ and $4$ simple zeroes on the Riemann sphere, then its ribbon graph $\Gamma$ has $4$ components and $4$ vertices of degrees $3$. The graph $\Gamma$ can only be Figure \ref{fig:4SimpleZeroes} by Lemma \ref{lem:PossibleGraphs}.

 In the following of this section, we give the proof of the remaining part of Theorem \ref{thm:GeneralCase}. In Proposition \ref{prop:BelyiFactor}, we prove that $f$ factors through another Belyi map with degree $\frac{1}{2}\deg f$. And then show that $f$ has minimal degree in Proposition \ref{prop:MinimalDegree}. We also give $5$ explicit examples by our construction at the end of this section.
\begin{prop}
\label{prop:BelyiFactor}
  Let $f(x): \mathbb{P}^{1} \rightarrow \mathbb{P}^{1}$ be a Belyi map satisfies
  \begin{itemize}
    \item $f^*q_0$ is Strebel differential;
    \item $f^*q_0$ has 4 simple zeroes;
    \item $f^*q_0$ has 4 double poles with the same residues.
  \end{itemize}
  Then there exists a Belyi map $g(x)$ so that $f = g \circ x^2$.
\end{prop}
\begin{proof}
  By the conditions of $f^*q_0$,  we know that the skeleton of ribbon graph of $f^*q_0$ is the Figure \ref{fig:4SimpleZeroes}. Since the residues of $f^*q_0$ are equal to each other, the local ramification degrees over $1$ of $f$ are the same to each other. As a result, $\deg f= 4d$ for some positive integer $d \geq 2$. The branch data over $0, \infty, 1$ of $f$ has only two possible cases for $d \geq 3$
  \begin{itemize}
    \item[$\blacklozenge$]  ($3^4,2^{2d-6}$), $2^{2d}$, $d^4$;
    \item[$\blacklozenge$]  ($3^2,2^{2d-3}$), ($3^2,2^{2d-3}$), $d^4$.
  \end{itemize}
  For the first case, the vertices of the dessin(inverse image of segment $[-\infty, 0]$) of $f$ have the same colour. We can draw the dessin as follows if the points on edges are omitted.
  
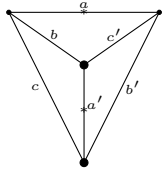
\begin{center}
    \begin{tikzpicture}
      \draw (0,0.3)--(1,1)
            (0,0.3)--(0,-1)
            (0,0.3)--(-1,1)
            (1,1)--(0,-1)
            (0,-1)--(-1,1)
            (-1,1)--(1,1);
       \filldraw (0,0.3) circle (1.5pt)
                 (0,-1) circle (1.5pt)
                 (1,1) circle (0.8pt)
                 (-1,1) circle (0.8pt);
      \node at (0,1.1) {\tiny{$a$}};
      \node at (0.15,-0.2) {\tiny{$a'$}};
      \node at (-0.4,0.7) {\tiny{$b$}};
      \node at (0.65,0) {\tiny{$b'$}};
      \node at (0.4,0.7) {\tiny{$c'$}};
      \node at (-0.65,0) {\tiny{$c$}};
      \node at (0,1) {\tiny{$*$}};
      \node at (0,-0.3) {\tiny{$*$}};
    \end{tikzpicture}
  \end{center}
  In order to guarantee the residues of $f^*q_0$ are equal to each other, the edges $a(b,c)$ and $a'(b',c',\text{ respectively})$ must be with the same coloured pointed. Assume that the colour of  middle point on edge $a$ is $*$(black or white), then the dessin of $f$ is the pull back by $x^2$ of
  
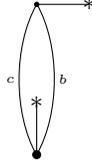
\begin{center}
    \begin{tikzpicture}
      \draw (0,-0.3)--(0,-1)
            (0,1)--(0.7,1);
      \filldraw (0,-1) circle (1.5pt)
                (0,1) circle (0.8pt);
      \draw (0,-1) arc (-45:45:0.8cm and 1.42cm);
      \draw (0,1) arc (135:225:0.8cm and 1.42cm);
      \node at (0.7,1) {$*$};
      \node at (0,-0.3) {$*$};
      \node at (-0.35,0) {\tiny{$c$}};
      \node at (0.35,0) {\tiny{$b$}};
    \end{tikzpicture}
  \end{center}
  For the second case, the proof is similar and Example \ref{ex:degree12} is an explicit construction.
\end{proof}

Suppose $q$ is the Strebel differential corresponding to the metric ribbon graph as Figure \ref{fig:Generalcase}. Let
\begin{align*}
  F_q = \{f: \mathbb{P}^{1} \rightarrow \mathbb{P}^{1} | f^*q_0 = \nu q \text{ for some nonzero complex number }\nu \}.
\end{align*}
We have the following proposition
\begin{prop}
\label{prop:MinimalDegree}
    For any given 3 positive rational numbers $a,b,c$ satisfying $a+b+c=1$, Let $d$ be the minimal positive integer so that $da,db,dc \in \mathbb{Z}$. Then GCD$(da,db,dc) = 1$ and
    \begin{itemize}
      \item $2\mid d$, $\min\limits_{f \in F_q} \deg f = 2d$;
      \item $2 \nmid d$,$\min\limits_{f \in F_q} \deg f = 4d$.
    \end{itemize}
\end{prop}
\begin{proof}
  If GCD$(da,db,dc) = k > 1 $, then $k|(da+db+dc) = d$. Let $d'=\frac{d}{k}$, we have $d'a,d'b,d'c \in \mathbb{Z}$. Contradiction!

  For the first case, there must be two odd numbers in $\{da,db,dc\}$ since GCD$(da,db,dc)=1$ and $da+db+dc=d(even)$. Without loss of generality, we assume $db$ and $dc$ are odd and $da$ is even. We can draw dessin as following
  
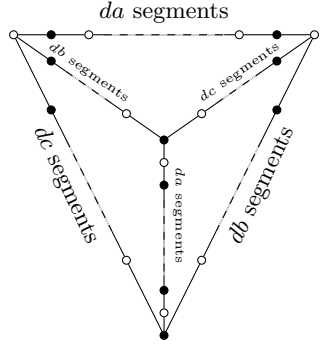
\begin{center}
    \begin{tikzpicture}
      \draw (0,0.6)--(2,2)
            (0,0.6)--(0,-2)
            (0,0.6)--(-2,2)
            (2,2)--(0,-2)
            (0,-2)--(-2,2)
            (-2,2)--(2,2);
      \draw[white] (0.7,1.09)--(1.3,1.51)
                   (-0.7,1.09)--(-1.3,1.51)
                   (0,-0.2)--(0,-1.2)
                   (-0.8,2)--(0.8,2)
                   (-1.3,0.6)--(-0.7,-0.6)
                   (1.3,0.6)--(0.7,-0.6);
      \draw[dashed] (0.7,1.09)--(1.3,1.51)
                    (-0.7,1.09)--(-1.3,1.51)
                    (0,-0.2)--(0,-1.2)
                    (-0.8,2)--(0.8,2)
                    (-1.3,0.6)--(-0.7,-0.6)
                    (1.3,0.6)--(0.7,-0.6);
      \filldraw[white] (-2,2) circle (1.5pt)
                       (-1,2) circle (1.5pt)
                       (1,2) circle (1.5pt)
                       (2,2) circle (1.5pt)
                       (-0.5,-1) circle (1.5pt)
                       (0.5,-1) circle (1.5pt)
                       (0,0.3) circle (1.5pt)
                       (0,-1.7) circle (1.5pt)
                       (-0.5,0.95) circle (1.5pt)
                       (0.5,0.95) circle (1.5pt);
      \draw (-2,2) circle (1.5pt)
            (-1,2) circle (1.5pt)
            (1,2) circle (1.5pt)
            (2,2) circle (1.5pt)
            (-0.5,-1) circle (1.5pt)
            (0.5,-1) circle (1.5pt)
            (0,0.3) circle (1.5pt)
            (0,-1.7) circle (1.5pt)
            (-0.5,0.95) circle (1.5pt)
            (0.5,0.95) circle (1.5pt);

      \filldraw (-1.5,2) circle (1.5pt)
                (1.5,2) circle (1.5pt)
                (-1.5,1) circle (1.5pt)
                (1.5,1) circle (1.5pt)
                (0,0.6) circle (1.5pt)
                (0,-2) circle (1.5pt)
                (0,0) circle (1.5pt)
                (0,-1.4) circle (1.5pt)
                (-1.5,1.65) circle (1.5pt)
                (1.5,1.65) circle (1.5pt);
      \node at (0,2.3) {\small{$da$ segments}};
      \node[rotate=-90] at (0.2,-0.4) {\tiny{$da$ segments}};
      \node[rotate=65]at (1.3,0) {\small{$db$ segments}};
      \node[rotate=-35] at (-1,1.5) {\tiny{$db$ segments}};
      \node[rotate=-65]at (-1.3,0) {\small{$dc$ segments}};
      \node[rotate=35] at (1,1.5) {\tiny{$dc$ segments}};
    \end{tikzpicture}
  \end{center}

  The degree of the corresponding Belyi map $f$ is $2d$ and $f \in F_q$. Suppose there exists $g \in F_q$ such that $\deg g = 2d' < 2d$, then each edge of dessin associated to $g$ has $d'a(d'b \text{ or }d'c)$ segments i.e. $d'a,d'b,d'c \in \mathbb{Z}$. Which has a contradiction with the minimality of $d$.

  For the second one, the possible parity of $(da,db,dc)$ is (even, even, odd) or (odd, odd, odd). Similarly, consider dessin d'enfant
  
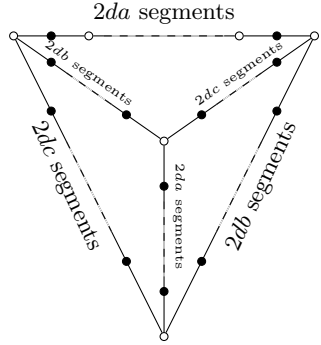
\begin{center}
    \begin{tikzpicture}
      \draw (0,0.6)--(2,2)
            (0,0.6)--(0,-2)
            (0,0.6)--(-2,2)
            (2,2)--(0,-2)
            (0,-2)--(-2,2)
            (-2,2)--(2,2);
      \draw[white] (0.7,1.09)--(1.3,1.51)
                   (-0.7,1.09)--(-1.3,1.51)
                   (0,-0.2)--(0,-1.2)
                   (-0.8,2)--(0.8,2)
                   (-1.3,0.6)--(-0.7,-0.6)
                   (1.3,0.6)--(0.7,-0.6);
      \draw[dashed] (0.7,1.09)--(1.3,1.51)
                    (-0.7,1.09)--(-1.3,1.51)
                    (0,-0.2)--(0,-1.2)
                    (-0.8,2)--(0.8,2)
                    (-1.3,0.6)--(-0.7,-0.6)
                    (1.3,0.6)--(0.7,-0.6);
      \filldraw[white] (-2,2) circle (1.5pt)
                       (-1,2) circle (1.5pt)
                       (1,2) circle (1.5pt)
                       (2,2) circle (1.5pt)
                       (0,0.6) circle (1.5pt)
                       (0,-2) circle (1.5pt);
      \draw (-2,2) circle (1.5pt)
            (-1,2) circle (1.5pt)
            (1,2) circle (1.5pt)
            (2,2) circle (1.5pt)
            (0,0.6) circle (1.5pt)
            (0,-2) circle (1.5pt);

      \filldraw (-1.5,2) circle (1.5pt)
                (1.5,2) circle (1.5pt)
                (-1.5,1) circle (1.5pt)
                (1.5,1) circle (1.5pt)
                (0,0) circle (1.5pt)
                (0,-1.4) circle (1.5pt)
                (-1.5,1.65) circle (1.5pt)
                (1.5,1.65) circle (1.5pt)
                (-0.5,-1) circle (1.5pt)
                (0.5,-1) circle (1.5pt)
                (-0.5,0.95) circle (1.5pt)
                (0.5,0.95) circle (1.5pt);
      \node at (0,2.3) {\small{$2da$ segments}};
      \node[rotate=-90] at (0.2,-0.4) {\tiny{$2da$ segments}};
      \node[rotate=65]at (1.3,0) {\small{$2db$ segments}};
      \node[rotate=-35] at (-1,1.5) {\tiny{$2db$ segments}};
      \node[rotate=-65]at (-1.3,0) {\small{$2dc$ segments}};
      \node[rotate=35] at (1,1.5) {\tiny{$2dc$ segments}};
    \end{tikzpicture}
  \end{center}
  Then the Belyi map $f$ associated to this dessin has degree $4d$ and $f \in F_q$. Since there does not exist  bicolour triangle such that the parity of the number of segments on $3$ edges is (even, even, odd) or (odd, odd, odd), $\deg f$ is minimal.
\end{proof}
At the very end of this section, we give some examples by our own method.
\begin{ex}
  Consider the following dessin d'enfant(can also be viewed as a metric ribbon graph with $(a,b,c) =(\frac{1}{2}, \frac{1}{4}, \frac{1}{4})$),
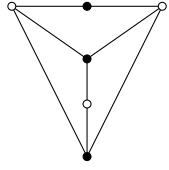
\begin{figure}[H]
  \begin{center}
    \begin{tikzpicture}
      \draw (0,0.3)--(1,1)
            (0,0.3)--(0,-1)
            (0,0.3)--(-1,1)
            (1,1)--(0,-1)
            (0,-1)--(-1,1)
            (-1,1)--(1,1);
      \fill[white] (1,1) circle (1.5pt)
                   (-1,1) circle (1.5pt)
                   (0,-0.3) circle (1.5pt);
      \draw  (1,1) circle (1.5pt)
             (-1,1) circle (1.5pt)
             (0,-0.3) circle (1.5pt);
      \filldraw (0,1) circle (1.5pt)
                (0,0.3) circle (1.5pt)
                (0,-1) circle (1.5pt);
    \end{tikzpicture}
  \end{center}
\caption{Dessin corresponds to a Belyi map of degree $8$.}
\end{figure}
The corresponding Belyi map is
\begin{equation*}
  f(x) = -\frac{1}{2^{12}}\frac{(x-1)^2(9x^2+14x+9)^3}{x^3(x+1)^2},
\end{equation*}
and the ramification degrees over $0,1,\infty$ are $(2,3,3),(2,2,2,2)$ and $(2,3,3)$ respectively. The points of $f^{-1}(1)$ satisfying equation
\begin{align*}
  27x^4 + 36x^3 + 2x^2 + 36x + 27 = 0,
\end{align*}
whose roots are
\begin{align*}
  x_1 &= -\frac{1}{3}\Big(\frac{4}{\sqrt{3}}+1+2\sqrt{\frac{2}{3}(\sqrt{3}-1)}\Big),
  &x_2 = -\frac{1}{3}\Big(\frac{4}{\sqrt{3}}+1-2\sqrt{\frac{2}{3}(\sqrt{3}-1)}\Big), \\
  x_3 &= \frac{1}{3}\Big(\frac{4}{\sqrt{3}}-1+2i\sqrt{\frac{2}{3}(\sqrt{3}+1)}\Big),
  &x_4 = \frac{1}{3}\Big(\frac{4}{\sqrt{3}}-1-2i\sqrt{\frac{2}{3}(\sqrt{3}+1)}\Big).
\end{align*}
The pull back of $q_0$ by $f$ is
\begin{align*}
  f^*q_0 = +\frac{dx^2}{4\pi^2}\frac{4096x(9x^2+14x+9)}{(27x^4 + 36x^3 + 2x^2 + 36x + 27)^2}.
\end{align*}
Then
\begin{align*}
  q_1 = +\frac{dx^2}{4\pi^2}\frac{1024x(9x^2+14x+9)}{(27x^4 + 36x^3 + 2x^2 + 36x + 27)^2}
\end{align*}
is a Strebel differential with $4$ simple zeroes and $4$ double poles and residue vector $(1,1,1,1)$. Consider the M\"{o}bius transformation $x \mapsto \frac{x-x_2}{x-x_3}\cdot\frac{x_1-x_3}{x_1-x_2}$, then $q_1$ becomes to
\begin{small}
  \begin{align*}
    -\frac{dx^2}{4\pi^2}\cdot\frac{x^4-(2+\sqrt{2}i)x^3-(\frac{1}{2}-\frac{3\sqrt{2}}{2}i)x^2 + (\frac{3}{2}-3\sqrt{2}i)x - \frac{23}{8}+\frac{5\sqrt{2}}{4}i}{x^2 (x-1)^2 (x-(\frac{1}{2}+\frac{5\sqrt{2}}{4}i))^2}.
  \end{align*}
\end{small}
Hence, $\lambda = \frac{1}{2}+\frac{5\sqrt{2}}{4}i$ and $\mu(\lambda) = 1+\frac{3\sqrt{2}}{2}i$.
\end{ex}

\begin{ex}
  Let us consider a Belyi map $f(x)$
  \begin{align*}
    \frac{f(x)}{1+f(x)} = -\frac{64x^3(x^3-1)^3}{(8x^3+1)^3},
  \end{align*}
  whose local ramification degrees over $ 1, 0, \infty$ are $(3,3,3,3), (3,3,3,3)$ and $(2,2,2,2,2,2)$ respectively.
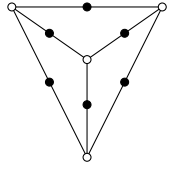
\begin{figure}[H]
  \begin{center}
    \begin{tikzpicture}
      \draw (0,0.3)--(1,1)
            (0,0.3)--(0,-1)
            (0,0.3)--(-1,1)
            (1,1)--(0,-1)
            (0,-1)--(-1,1)
            (-1,1)--(1,1);
      \fill[white] (1,1) circle (1.5pt)
                   (-1,1) circle (1.5pt)
                   (0,0.3) circle (1.5pt)
                   (0,-1) circle (1.5pt);
      \draw  (1,1) circle (1.5pt)
             (-1,1) circle (1.5pt)
             (0,0.3) circle (1.5pt)
             (0,-1) circle (1.5pt);
      \filldraw (0,1) circle (1.5pt)
                (0.5,0) circle (1.5pt)
                (-0.5,0) circle (1.5pt)
                (0,-0.3) circle (1.5pt)
                (0.5,0.65) circle (1.5pt)
                (-0.5,0.65) circle (1.5pt);
     \end{tikzpicture}
  \end{center}
\caption{Dessin corresponds to the Belyi map of degree $12$.}
\end{figure}
  Similarly, we can get a Strebel differential
  \begin{align*}
    -\frac{dx^2}{4\pi^2}\cdot \frac{x^4-(2-\frac{2}{\sqrt{3}}i)x^3+(1-\sqrt{3}i)x^2+\frac{4i}{\sqrt{3}}x - (\frac{1}{2}+\frac{\sqrt{3}}{2}i)}{x^2 (x-1)^2 (x-(\frac{1}{2}-\frac{\sqrt{3}}{2}i))^2}.
  \end{align*}
  Hence $\mu (\frac{1}{2}-\frac{\sqrt{3}}{2}i) = 1 - \frac{\sqrt{3}}{3}i$. In fact, we have a simpler way to figure out $\mu (\frac{1}{2}-\frac{\sqrt{3}}{2}i)$. Note that $1-\lambda = 1/\lambda$ when $\lambda = \frac{1}{2}-\frac{\sqrt{3}}{2}i $, then $2 - \mu(\lambda) = \frac{\mu(\lambda)}{\lambda}$ which implies $\mu(\lambda) =  1 - \frac{\sqrt{3}}{3}i$.
\end{ex}

\begin{ex}
\label{ex:degree12}
  If $\deg f=12$ and $\frac{1}{3^2}\cdot f^*q_0$ has $4$ simple zeroes and $4$ double poles with residue vector $(1,1,1,1)$, the local ramification degrees over $1,0, \infty$ can also be $\Big(3^4, (2^3,3^2), (2^3,3^2) \Big)$. The only possible Dessins d'Enfants of $\frac{f}{f-1}$ are
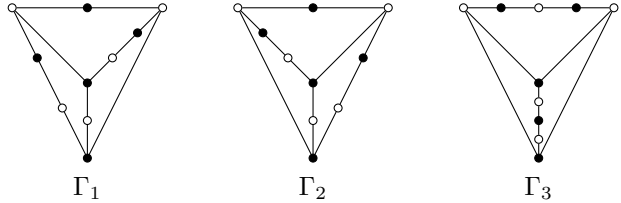
\begin{figure}[H]
  \begin{center}
    \begin{tikzpicture}
      \draw (-3,0)--(-2,1)  
          (-3,0)--(-4,1)
          (-3,0)--(-3,-1)
          (-2,1)--(-4,1)
          (-4,1)--(-3,-1)
          (-3,-1)--(-2,1);
      \filldraw [white] (-4,1) circle (1.5pt)
                      (-2,1) circle (1.5pt)
                      (-3,-0.5) circle (1.5pt)
                      (-10/3,-1/3) circle (1.5pt)
                      (-8/3,1/3) circle (1.5pt);
      \draw (-4,1) circle (1.5pt)
          (-2,1) circle (1.5pt)
          (-3,-0.5) circle (1.5pt)
          (-10/3,-1/3) circle (1.5pt)
          (-8/3,1/3) circle (1.5pt);
      \filldraw (-3,0) circle (1.5pt)
                (-3,1) circle (1.5pt)
              (-3,-1) circle (1.5pt)
              (-11/3,1/3) circle (1.5pt)
              (-7/3,2/3) circle (1.5pt);
      \node at (-3,-1.4) {$\Gamma _1$};

      \draw (0,0)--(1,1) 
          (0,0)--(-1,1)
          (0,0)--(0,-1)
          (1,1)--(-1,1)
          (-1,1)--(0,-1)
          (0,-1)--(1,1);
      \filldraw [white] (-1,1) circle (1.5pt)
                      (1,1) circle (1.5pt)
                      (0,-0.5) circle (1.5pt)
                      (1/3,-1/3) circle (1.5pt)
                      (-1/3,1/3) circle (1.5pt);
      \draw (-1,1) circle (1.5pt)
          (1,1) circle (1.5pt)
          (0,-0.5) circle (1.5pt)
          (1/3,-1/3) circle (1.5pt)
          (-1/3,1/3) circle (1.5pt);
      \filldraw (0,0) circle (1.5pt)
                (0,1) circle (1.5pt)
              (0,-1) circle (1.5pt)
              (2/3,1/3) circle (1.5pt)
              (-2/3,2/3) circle (1.5pt);
      \node at (0,-1.4) {$\Gamma _2$};

      \draw (3,0)--(4,1)  
          (3,0)--(2,1)
          (3,0)--(3,-1)
          (4,1)--(2,1)
          (2,1)--(3,-1)
          (3,-1)--(4,1);
      \filldraw [white] (2,1) circle (1.5pt)
                      (4,1) circle (1.5pt)
                      (3,1) circle (1.5pt)
                      (3,-1/4) circle (1.5pt)
                      (3,-3/4) circle (1.5pt);
      \draw (2,1) circle (1.5pt)
          (4,1) circle (1.5pt)
          (3,1) circle (1.5pt)
          (3,-1/4) circle (1.5pt)
          (3,-3/4) circle (1.5pt);
      \filldraw (3,0) circle (1.5pt)
              (3,-1) circle (1.5pt)
              (3,-1/2) circle (1.5pt)
              (2.5,1) circle (1.5pt)
              (7/2,1) circle (1.5pt);
      \node at (3,-1.4) {$\Gamma _3$};
    \end{tikzpicture}
  \end{center}
\caption{The possible ribbon graph of $f^*q_0$.}
\label{fig:AllPossible}
\end{figure}
  Which can also be viewed as metric ribbon graph with $(a,b,c) = (\frac{1}{3}, \frac{1}{6}, \frac{1}{2})$, $(\frac{1}{3}, \frac{1}{2}, \frac{1}{6})$, $(\frac{2}{3}, \frac{1}{6}, \frac{1}{6})$ respectively. In order to write down some explicit Belyi map with the above branch date, by Proposition \ref{prop:BelyiFactor}, we only need to construct a Belyi map $g(x)$ of degree $6$ with local ramification degrees $\Big((1,2,3), (1,2,3), (3,3) \Big)$. Up to a scalar factor, $g(x)$ has the form of
  \begin{equation*}
  g(x)= \frac{x^3(x-a_1)(x-a_2)^2}{(x-a_3)(x-a_4)^2},
\end{equation*}
and $f(x)= g(\frac{a_3 x+a_1}{x+1}) \circ x^2 = g \circ \frac{a_3x^2+a_1}{x^2+1}$ up to scale. The derivative of $g(x)$ is
\begin{equation*}
  g'(x)= \frac{x^2 (x-a_2)}{(x-a_3)^2(x-a_4)^3}\cdot h(x),
\end{equation*}
where
\begin{footnotesize}
  \begin{align*}
    h(x)= &3x^4 - (2a_1+a_2+4a_3+5a_4)x^3 + (3a_1a_3+2a_2a_3+4a_1a_4+3a_2a_4+6a_3a_4)x^2\\
          &-(a_1a_2a_3+2a_1a_2a_4+5a_1a_3a_4+4a_2a_3a_4)x + 3a_1a_2a_3a_4.
  \end{align*}
\end{footnotesize}
Since $h(x)$ has $2$ double zeroes, we may assume that $h(x)=3(x-1)^2(x-a_5)^2$ and $g(1)=g(a_5)$. By comparing the coefficients of $h(x)$, we find $a_5$ satisfies the following equation
\begin{equation*}
  t^6+6t^5+15t^4+36t^3+15t^2+6t+1=0
\end{equation*}
and we have
\begin{align*}
  a_1 &= -\frac{1}{2}(1+10a_5+35a_5^2+15a_5^3+6a_5^4+a_5^5), \\
  a_2 &= \frac{1}{4}(7+16a_5+35a_5^2+15a_5^3+6a_5^4+a_5^5), \\
  a_3 &= \frac{1}{16}(21+45a_5+166a_5^2+70a_5^3+29a_5^4+5a_5^5), \\
  a_4 &= \frac{1}{20}(3a_5-61a_5^2-25a_5^3-11a_5^4-2a_5^5).
\end{align*}
{\bf Claim: The polynomial $P(t)=t^6+6t^5+15t^4+36t^3+15t^2+6t+1 \in \mathbb{Q}[t]$ is irreducible.}
\begin{proof}
  Note that $P(t)=t^6+6t^5+15t^4+36t^3+15t^2+6t+1=(t+1)^6+16 t^3$. Hence
  \begin{align*}
    P(t) &= (t+1)^6 = (t^2+1)(t^4+1) \text{ in } \mathbb{F}_2[t], \\
    P(t) &= t^6+1 = (t^2+1)^3 = (t^2+1)(t^4+2 t^2+1) \text{ in } \mathbb{F}_3[t].
  \end{align*}
  Suppose that $P(t)$ is reducible in $\mathbb{Z}[t]$. Since $t^2+1$ is irreducible in $\mathbb{F}_3[t]$, we have a polynomial factorization $P(t)=P_1(t)P_2(t)$ in $\mathbb{Z}[t]$, where $\deg P_1(t)=2,\deg P_2(t)=4$. Hence, we can assume that
  \begin{align*}
    P_1(t) &= t^2 + 6c_0 t + 1, \\
    P_2(t) &= t^4 + 6d_0 t^3 + (2+6d_1)t^2 + 6d_2 t + 1.
  \end{align*}
  The coefficients of $t$ and $t^5$ of $P_1(t)P_2(t)$ are $6(c_0+d_2)$ and $6(c_0+d_0)$ respectively, which imply that $d_0=d_2$ and $c_0+d_0 = 1$. The coefficient of $t^2$ is $3+6d_1+36c_0d_0 = 15$, therefore $c_0d_0=0$ and $d_1=2$. The only two possible factors are
  \begin{align*}
    \text{case }&1\\
    &P_1(t) = t^2 + 6t + 1\\
    &P_2(t) = t^4 + 14t^2 + 1\\
    \text{case }&2\\
    &P_1(t) = t^2 + 1\\
    &P_2(t) = t^4 + 6t^3 + 14t^2 + 6t+ 1
  \end{align*}
  However, both are impossible to satisfy $P(t) = P_1(t) P_2(t)$. By Gauss lemma, we know that $P(t)$ in irreducible in $\mathbb{Q}[t]$.
\end{proof}
As a corollary of this claim, we know that $a_1,a_2,a_3,a_4,a_5,0, 1$ are pairwise distinct. In order to get concrete expression of $f(x)$, we only need to solve the equation $P(t)=0$. Luckily, $P(t)$ is solvable by
radicals.
\begin{align*}
  (t+1)^6=-16t^3 \Longrightarrow (t+1)^2=e^{\frac{i\pi}{3}+\frac{2i\pi}{3}\cdot k}2^{\frac{4}{3}}t (k=0,1,2).
\end{align*}
We get the roots
\begin{align*}
  &\frac{1}{2}\Big(-2+2^{1/3}+2^{1/3}\sqrt{3}i \pm \sqrt{-4+(2-2^{1/3}-2^{1/3}\sqrt{3}i)^2}\Big), \\
  &-1-2^{1/3} \pm 2^{1/3}\sqrt{1+2^{2/3}}, \\
  &\frac{1}{2}\Big(-2+2^{1/3}-2^{1/3}\sqrt{3}i \pm \sqrt{-4+(2-2^{1/3}+2^{1/3}\sqrt{3}i)^2}\Big).
\end{align*}
By the construction we know that if $(a_1,a_2,a_3,a_4,a_5)$ is a solution of $g(x)$, then $(\frac{a_1}{a_5} , \frac{a_2}{a_5} , \frac{a_3}{a_5} , \frac{a_4}{a_5},\frac{1}{a_5})$ is also a solution of $g(x)$ and these two Belyi maps are equivalent under M\"{o}bius transformation. Hence, we only need to consider
\begin{align*}
  t_0 &= \frac{1}{2}\Big(-2+2^{1/3}+2^{1/3}\sqrt{3}i - \sqrt{-4+(2-2^{1/3}-2^{1/3}\sqrt{3}i)^2}\Big), \\
  t_1 &= -1-2^{1/3} - 2^{1/3}\sqrt{1+2^{2/3}}, \\
  t_2 &= \frac{1}{2}\Big(-2+2^{1/3}-2^{1/3}\sqrt{3}i - \sqrt{-4+(2-2^{1/3}+2^{1/3}\sqrt{3}i)^2}\Big).
\end{align*}
As before, we can figure out exact values of $\lambda$ and $\mu$. For example, the expression of $\lambda$ corresponding to $t_1$ is
\begin{small}
  \begin{align*}
    \frac{\Big(2 i + 2^{11/12} \sqrt{\frac{
   5\times 2^{1/6} \sqrt{2 + 2^{1/3}} (2 + 2^{2/3}) +
    2 \big(11 + 7\times 2^{1/3} + 7\times 2^{2/3} + 5 \sqrt{2 (2 + 2^{1/3})}\big)}{
   2^{1/6} (-18 + 6\times 2^{1/3} + 5\times 2^{2/3}) +
    \sqrt{2 + 2^{1/3}} (-10 - 2\times 2^{1/3} + 9\times 2^{2/3})}}\Big)^2}
    {\Big(-2 i + 2^{11/12} \sqrt{\frac{
   5\times 2^{1/6} \sqrt{2 + 2^{1/3}} (2 + 2^{2/3}) +
    2 \big(11 + 7\times 2^{1/3} + 7\times 2^{2/3} + 5 \sqrt{2 (2 + 2^{1/3})}\big)}{
   2^{1/6} (-18 + 6\times 2^{1/3} + 5\times 2^{2/3}) +
    \sqrt{2 + 2^{1/3}} (-10 - 2\times 2^{1/3} + 9\times 2^{2/3})}}\Big)^2},
  \end{align*}
\end{small}
which is too complicated. Here we give the  approximate values of $\lambda$ and $\mu$ correspond to each $t_k$
\begin{align*}
  \lambda_0 &= 1.3157 - 1.5429 i,     &\mu_0 = 1.6586 - 1.87049 i; \\
  \lambda_1 &= 0.9726 + 0.2324 i,     &\mu_1 = 0.3689 + 0.04346 i; \\
  \lambda_2 &= 1.3157 + 1.5429 i,     &\mu_2 = 1.6586 + 1.87049 i.
\end{align*}
Now we want to give the correspondence between $\lambda_i$ and the ribbon graphs of Figure \ref{fig:AllPossible}. At first, we note that the branch data of $g(x)$ is $(\{1,2,3\}, \{1,2,3\}, \{3,3\})$. Up to colour exchange and isomorphism there are $3$ possible dessins of $g(x)$.
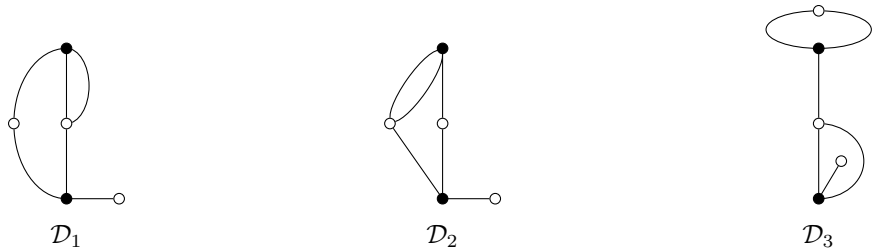
\begin{figure}[H]
\begin{center}
  \begin{tikzpicture}
    \draw (-5,0)--(-5,1)     
          (-5,0)--(-5,-1)
          (-5,-1)--(-4.3,-1);
    \draw (-5,1) arc (90:270:0.7cm and 1cm);
    \draw (-5,0) arc (-90:90:0.3cm and 0.5cm);
    \filldraw (-5,1) circle (2pt)
              (-5,-1) circle (2pt);
    \filldraw [white] (-5,0) circle (2pt)
                      (-4.3,-1) circle (2pt)
                      (-5.7,0) circle (2pt);
    \draw (-5,0) circle (2pt)
          (-4.3,-1) circle (2pt)
          (-5.7,0) circle (2pt);
    \node at (-5,-1.5) {$\mathcal{D}_1$};

    \draw (0,0)--(0,1)    
          (0,0)--(0,-1)
          (0,-1)--(0.7,-1)
          (0,-1)--(-0.7,0);
    \draw[rotate around={145:(-0.35,0.5)}] (-0.35,0.5) ellipse (0.15cm and 0.57cm);
    \filldraw (0,1) circle (2pt)
              (0,-1) circle (2pt);
    \filldraw [white] (0,0) circle (2pt)
                      (0.7,-1) circle (2pt)
                      (-0.7,0) circle (2pt);
    \draw (0,0) circle (2pt)
          (0.7,-1) circle (2pt)
          (-0.7,0) circle (2pt);
    \node at (0,-1.5) {$\mathcal{D}_2$};

    \draw (5,0)--(5,1)   
          (5,0)--(5,-1)
          (5,-1)--(5.3,-0.5);
    \draw (5,1.25) ellipse (0.7cm and 0.25cm);
    \draw (5,-1) arc (-90:90:0.6cm and 0.5cm);
    \filldraw (5,1) circle (2pt)
              (5,-1) circle (2pt);
    \filldraw [white] (5,0) circle (2pt)
                      (5.3,-0.5) circle (2pt)
                      (5,1.5) circle (2pt);
    \draw (5,0) circle (2pt)
          (5.3,-0.5) circle (2pt)
          (5,1.5) circle (2pt);
    \node at (5,-1.5) {$\mathcal{D}_3$};
  \end{tikzpicture}
\end{center}
\caption{The  possible dessins of $g(x)$.}
\end{figure}
By definition, ribbon graph of $f^*q_0$ is the inverse image of $[0,\infty]$(negative real axis) by $f(x)$ and the dessin of $g(x)$ is the inverse image of segment $[0,1]$ by $g$. In order to get the ribbon graphs correspond to dessins of $g(x)$. We need to construct "dual graph" of these dessins. For example, consider the dessin $\mathcal{D}_2$,
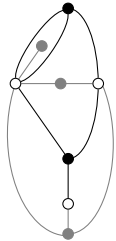
\begin{figure}[H]
\begin{center}
  \begin{tikzpicture}
    \draw (0,-1)--(0,-1.6)   
          (-0,-1)--(-0.7,0)
          (0,-1) arc (-90:90:0.4cm and 1cm);
    \draw[rotate around={145:(-0.35,0.5)}] (-0.35,0.5) ellipse (0.15cm and 0.57cm);
    \draw[gray] (-0.7,0)--(0.4,0)
                (-0.7,0)--(-0.35,0.5)
                (0,-1.6)--(0,-2)
                (0,-2) arc (-90:45:0.6cm and 1.2cm)
                (-0.7,0) arc (150:270:0.8cm and 1.35cm);
    \filldraw (0,1) circle (2pt)
              (0,-1) circle (2pt);
    \filldraw [white] (0.4,0) circle (2pt)
                      (0,-1.6) circle (2pt)
                      (-0.7,0) circle (2pt);
    \draw (0.4,0) circle (2pt)
          (0,-1.6) circle (2pt)
          (-0.7,0) circle (2pt);
    \filldraw[gray] (-0.35,0.5) circle (2pt)
               (0,-2) circle (2pt)
               (-0.1,0) circle (2pt);
  \end{tikzpicture}
\end{center}
\caption{Construction new dessin from old one.}
\end{figure}
then, the corresponding ribbon graph can be constructed as following
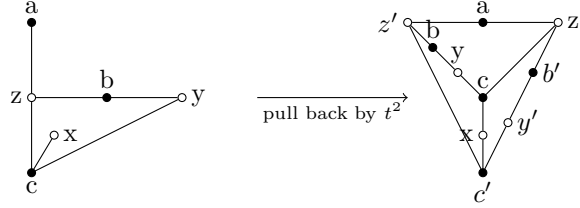
\begin{figure}[H]
\begin{center}
  \begin{tikzpicture}
    \draw (-4,0)--(-5,0) node[above]{b}
          (-5,0)--(-6,0) node[left]{z}
          (-6,0)--(-6,1) node[above]{a}
          (-6,0)--(-6,-1) node[below]{c}
          (-6,-1)--(-5.7,-0.5) node[right]{x}
          (-6,-1)--(-4,0) node[right]{y};
    \filldraw (-5,0) circle (1.5pt)
              (-6,1) circle (1.5pt)
              (-6,-1) circle (1.5pt);
    \filldraw[white] (-6,0) circle (1.5pt)
                     (-5.7,-0.5) circle (1.5pt)
                     (-4,0) circle (1.5pt);
    \draw (-6,0) circle (1.5pt)
          (-5.7,-0.5) circle (1.5pt)
          (-4,0) circle (1.5pt);

    \draw (1,1)--(0,0) node[above]{c}
          (0,0)--(-1/3,1/3)node[above]{y}--(-2/3,2/3)node[above]{b}--(-1,1) node[left]{$z'$}
          (0,0)--(0,-0.5)node[left]{x}--(0,-1) node[below]{$c'$}
          (-1,1)--(0,1)node[above]{a}--(1,1) node[right]{z}
          (-1,1)--(0,-1)
          (1,1)--(2/3,1/3)node[right]{$b'$}--(1/3,-1/3)node[right]{$y'$}--(0,-1);
      \filldraw [white] (-1,1) circle (1.5pt)
                      (1,1) circle (1.5pt)
                      (0,-0.5) circle (1.5pt)
                      (1/3,-1/3) circle (1.5pt)
                      (-1/3,1/3) circle (1.5pt);
      \draw (-1,1) circle (1.5pt)
          (1,1) circle (1.5pt)
          (0,-0.5) circle (1.5pt)
          (1/3,-1/3) circle (1.5pt)
          (-1/3,1/3) circle (1.5pt);
      \filldraw (0,0) circle (1.5pt)
                (0,1) circle (1.5pt)
              (0,-1) circle (1.5pt)
              (2/3,1/3) circle (1.5pt)
              (-2/3,2/3) circle (1.5pt);
      \draw[->] (-3,0) -- (-1,0) node at (-2,-0.2){\scriptsize{pull back by $t^2$}};
  \end{tikzpicture}
\end{center}
\caption{The ribbon graph associated to dessin.}
\end{figure}
By the same way, we can construct ribbon graphs corresponding to the other two dessins. In summary, we have
\begin{align*}
  \mathcal{D}_1&\leftrightarrow \Gamma_1 \\
  \mathcal{D}_2&\leftrightarrow \Gamma_2 \\
  \mathcal{D}_3&\leftrightarrow \Gamma_3.
\end{align*}
We first observe that the ribbon graph $\Gamma_1$ is the image of the ribbon graph $\Gamma_2$ by the complex conjugation $z\mapsto \overline{z}$, an orientation reversing homeomorphism. On the other hand, the equation $(t+1)^2=-2^{\frac{4}{3}}t$ is fixed by complex conjugation. Hence, the corresponding ribbon graph of $(\lambda_1, \mu_1)$ is $\Gamma_3$. By directly computing the inverse image of the segment $[0,1]$ by $g$, we know that the corresponding dessin of $(\lambda_0, \mu_0)$ is $\mathcal{D}_1$, i.e.
\begin{align*}
  (\lambda_0,\mu_0) &\leftrightarrow \Gamma_1 \\
  (\lambda_1,\mu_1) &\leftrightarrow \Gamma_3 \\
  (\lambda_2,\mu_2) &\leftrightarrow \Gamma_2.
\end{align*}
\end{ex}
By the above examples and Proposition \ref{prop:BelyiFactor} and \ref{prop:MinimalDegree}, we complete the proof of Theorem \ref{thm:GeneralCase}.

\section*{Acknowledgments}
The authors would like to express their deep gratitude to Mr Bo Li and Yiran Cheng for their valuable discussions during the course of this work. The first author also wants to thank Professor Zheng Hua at University of Hong Kong for his invitation and hospitality.
The first author is partially supported by National Natural Science Foundation of China (grant no. 11471298, no. 11622109 and no. 11721101) and the Fundamental Research Funds for the Central Universities. The second author is supported in part by the National Natural Science Foundation of China (Grant No. 11571330) and the Fundamental Research Funds for the Central Universities.

\small {\sc Bin Xu\\
Wu Wen-Tsun Key Laboratory of Math, USTC, CAS\\
School of Mathematical Sciences\\
University of Science and Technology of China\\
Hefei 230026 China}\\
bxu@ustc.edu.cn\\

{\sc  Jijian Song\\
Center for Applied Mathematics\\
School of Mathematics, Tianjin University\\
Tianjin 300350 China}\\
\Envelope smath@mail.ustc.edu.cn


\begin{thebibliography}{99}
 \bibitem{Arbarello2010} Enrico Arbarello and Maurizio Cornalba, \emph{Complex variable function---Jenkins-Strebel differentials}, Rendiconti Lincei-Matematica e Applicazioni, {\bf 21}(2010), 115-157.
  \bibitem{Harer1988} John L. Harer, \emph{The cohomology of the moduli space of curves}, Theory of Moduli(Montecatini Terme, 1985), Lecture Notes in Mathematics, vol. 1337, Springer, Berlin, 1988, pp. 138-221.
 \bibitem {Jenkins1957} James A. Jenkins, \emph{On the existence of certain general extremal metrics}, Annals of Mathematics, {\bf 66}(1957), 440-453.
 \bibitem{Kon1992} Maxim Kontsevich, \emph{Intersection theory on the moduli space of curves and the matrix Airy function}, Communications in Mathematical Physics, {\bf 147}(1992), 1-23.
 \bibitem{Loo1995} Eduard Looijenga, \emph{Cellular decompositions of compactified moduli spaces of pointed curves}, The Moduli Space of Curves (R. Dijkgraaf et al., eds), Birkh\"{a}user, Basel, 1995, 369-400.
 \bibitem{Masur2009} Howard Masur, \emph{Geometry of Teichm\"{u}ller space with the Teichm\"{u}ller metric}, Surveys in Differential Geometry, {\bf 14}(2009), 295-314.
 \bibitem{Mulase1998} Motohico Mulase and Michael Penkava, \emph{Ribbon graphs, quadratic differentials on Riemann surfaces, and algebraic curves defined over $\overline{\mathbb{Q}}$}, Asian Journal of Mathematics, {\bf 2}(1998), 875-920.
 \bibitem{Mulase2001} Motohico Mulase and Michael Penkava, \emph{Periods of Strebel differentials and algebraic curves defined over the field of algebraic numbers}, Asian Journal of Mathematics, {\bf 6}(2002), 743-748.
 \bibitem{Song2017}Jijian Song, Yiran Cheng, Bo Li and Bin Xu, \emph{Drawing cone spherical metrics via Strebel differentials}, International Mathematics Research Notices, rny103, https://doi.org/10.1093/imrn/rny103.
 \bibitem {Strebel1984} Kurt Strebel, \emph{Quadratic differentials}, Springer-Verlag, 1984.
 \bibitem{Zvo2004} Dimitri Zvonkine, \emph{Strebel differentials on stable curves and Kontsevich's proof of Witten's conjecture}, arXiv:math/0209071.
\end{thebibliography}
\end{document}